\documentclass[11pt,reqno]{amsart}

\usepackage{amssymb,amsmath,amsthm,amscd,latexsym,amsfonts}
\usepackage{mathtools}
\usepackage[T1]{fontenc}

\usepackage{graphicx}
\usepackage{color}
\usepackage{cite}

\usepackage[a4paper,top=3cm, bottom=3cm, left=3cm, right=3cm]{geometry}

\newtheorem{thm}[subsection]{Theorem}

\newtheorem{lemma}[subsection]{Lemma}
\newtheorem{pro}[subsection]{Proposition}

\theoremstyle{definition}
\newtheorem{defn}[subsection]{Definition}
\newtheorem{ex}[subsection]{Example}
\theoremstyle{remark}
\newtheorem{rk}[subsection]{Remark}

\numberwithin{equation}{section} 

\newcommand{\bea}{\begin{eqnarray}}
\newcommand{\eea}{\end{eqnarray}}

\newcommand{\R}{\mathbb{R}}

\DeclareMathOperator{\Orb}{Orb}

\def\R{\mathbb{R}}

\begin{document}
\title [Algebras of cubic matrices]
{Algebras of cubic matrices}

\author {M. Ladra, U.A. Rozikov}

\address{M.\ Ladra\\ Department of Algebra, University of Santiago de Compostela, 15782, Spain.}
 \email {manuel.ladra@usc.es}

 \address{U.\ A.\ Rozikov\\ Institute of mathematics, 29, Do'rmon Yo'li str., 100125,
Tashkent, Uzbekistan.} \email {rozikovu@yandex.ru}

\begin{abstract} We consider algebras of $m\times m\times m$-cubic matrices (with $m=1,2,\dots$).
Since there are several kinds of multiplications of cubic matrices,
one has to specify a multiplication first and then define an algebra of cubic matrices (ACM)
with respect to this multiplication. We mainly use the associative multiplications
 introduced by Maksimov.
Such a multiplication depends on an associative binary operation on the set of size $m$.
We introduce a notion of equivalent operations and show that such operations generate isomorphic ACMs.
It is shown that an ACM is not baric. An ACM is commutative iff $m=1$.
We introduce a notion of accompanying algebra (which is $m^2$-dimensional) and show that
there is a homomorphism from any ACM to the accompanying algebra. We describe (left and right) symmetric operations and give
left and right zero divisors of the corresponding ACMs. Moreover several subalgebras and ideals of an ACM are constructed.
\end{abstract}



\maketitle

\section{Introduction}

Cubic matrices arise as a matrix of structural constants in a finite-dimensional algebra.
This cubic matrix generates an evolution quadratic operator in  dynamical systems
modeling in different disciplines, such as population dynamics \cite{Lyu}, physics \cite{UdRa}, economy \cite{Doh}.

In \cite{Eth1,Eth3} Etherington introduced the
formal language of abstract algebra to the study of genetics.
In recent years many
authors have tried to investigate algebras arising in
natural sciences (see for example \cite{Lyu,Reed,Tian,Wor}).

The algebra of quadratic (square) matrices is well studied (see for example \cite{Lam}).
To the best of our knowledge, there are few articles with some results on the algebras of cubic matrices.
In this paper we will give a systematic study of the algebras of cubic matrices.

The paper is organized as follows. Section~\ref{S:definitions}
is devoted to main definitions and remarks.  We define algebras
of $m\times m\times m$-cubic matrices (with $m=1,2,\dots$).
Since there are several kinds of multiplications of cubic matrices,
one has to specify a multiplication first and then define an algebra of cubic matrices (ACM)
with respect to this multiplication. In this paper we mainly use associative multiplications
which were introduced by Maksimov \cite{Mak}.
Such  a multiplication depends on an associative binary operation on the set of size $m$.
In Section~\ref{S:iso} we introduce a notion of equivalent operations and show that such operations generate isomorphic ACMs.
The result of Section~\ref{S:baric} is that an ACM is not baric.
In Section~\ref{S:accom} we introduce a notion of accompanying algebra
(which is $m^2$-dimensional) and show that there is a homomorphism
from any ACM to the accompanying algebra.
 In Section~\ref{S:comm} we show that an
ACM is commutative iff $m=1$, describe (left and right) symmetric operations and give
left and right zero divisors of the corresponding ACMs. Finally, Section~\ref{S:subal} is devoted
to the construction of subalgebras and ideals of an ACM.

\section{Definitions}\label{S:definitions}

Let $F$ be a field, and let $\mathcal A$ be a vector space over $F$ equipped with an additional
binary operation from $\mathcal A\times \mathcal A$ to $\mathcal A$, denoted
by $\cdot$ (we  will simply write $xy$ instead $x\cdot y$).  Then $\mathcal A$ is an \emph{algebra} over $F$ if the following identities
hold for every  $x, y$ and $z$ in $\mathcal A$, and every  $\lambda$ and $\mu$ in $F$:
 $(x + y)z = xz + yz$, $x(y + z) = xy + xz$, $(\lambda x)(\mu y) = (\lambda \mu) (xy)$.

For algebras over a field, the bilinear multiplication from $\mathcal A\times \mathcal A$ to $\mathcal A$,
is completely determined by the multiplication of the basis elements of $\mathcal A$. Conversely, once a basis for $\mathcal A$
 has been chosen, the products of the basis elements can be set arbitrarily, and then extended in a unique
 way to a bilinear operator on $\mathcal A$, i.e., the resulting multiplication satisfies the algebra laws.

A homomorphism between two algebras, $\mathcal A$ and $\mathcal B$, over a field $F$,
is a map $f \colon \mathcal A\rightarrow \mathcal B$ such that for all $\lambda \in F$ and $x,y\in\mathcal A$,
\[f(\lambda x) = \lambda f(x); \qquad  f(x + y) = f(x) + f(y); \qquad  f(xy) = f(x)f(y).\]
If $f$ is bijective then $f$ is said to be an isomorphism between $\mathcal A$ and $\mathcal B$.

Given a field $F$, any finite-dimensional algebra can be specified up
to isomorphism by giving its dimension (say $m$), and specifying $m^3$ structure constants $c_{ijk}$, which are scalars.
These structure constants determine the multiplication in $\mathcal A$ via the following rule:
\[{e}_{i} {e}_{j} = \sum_{k=1}^m c_{ijk} {e}_{k},\]
where $e_1,\dots,e_m$ form a basis of $\mathcal A$.

Thus the multiplication of a finite-dimensional algebra is given by a cubic matrix $(c_{ijk})$.
Following \cite{BLZ,Mak,Pan}
recall a notion of cubic matrix and different associative multiplication rules of cubic matrices:
a cubic matrix $Q=(q_{ijk})_{i,j,k=1}^m$ is a $m^3$-dimensional vector which can be uniquely written as
\[Q=\sum_{i,j,k=1}^m   q_{ijk}E_{ijk},\]
where $E_{ijk}$ denotes the cubic unit (basis) matrix, i.e. $E_{ijk}$ is a $m^3$- cubic matrix whose
$(i,j,k)$th entry is equal to 1 and all other entries are equal to 0.

Denote by $\mathfrak C$ the set of all cubic matrices over a field $F$. Then $\mathfrak C$ is an
$m^3$-dimensional vector space over $F$, i.e. for any matrices $A=(a_{ijk}), B=(b_{ijk})\in \mathfrak C$, $\lambda\in F$, we have
\[ A+B \coloneqq (a_{ijk}+b_{ijk})\in \mathfrak C, \qquad  \lambda A \coloneqq (\lambda a_{ijk})\in \mathfrak C. \]
Denote $I=\{1,2,\dots,m\}$.

Following  \cite{Mak} define the following multiplications for basis matrices $E_{ijk}$:
\begin{equation}\label{ma}
 E_{ijk}*_a E_{lnr}=\delta_{kl}E_{ia(j,n)r},
 \end{equation}
  where $a \colon   I\times I\to I$, $ (j,n) \mapsto a(j,n) \in I$,  is an arbitrary associative binary operation (see Remark \ref{re1} below for associative binary operations) and
  $\delta_{kl}$ is the Kronecker symbol.

  Denote by $\mathcal O_m$ the set of all associative binary operations on $I$.

 The general formula for the multiplication is the extension of \eqref{ma} by bilinearity, i.e.
 for any two cubic matrices $A=(a_{ijk}), B=(b_{ijk})\in \mathfrak C$
 the matrix $A*_a B=(c_{ijk})$ is defined by
 \begin{equation}\label{AB}
 c_{ijr}=\sum_{l,n: \, a(l,n)=j}\sum_k a_{ilk}b_{knr}.
 \end{equation}

 Denote by $\mathfrak C_a\equiv \mathfrak C_a^m=(\mathfrak C, *_a) $, $a\in\mathcal O_m$, the algebra of
 cubic matrices (ACM) given by the multiplication $*_a$.

 \begin{rk}Recently in \cite{BLZ} the authors provided 15 associative multiplication rules of cubic matrices,
 some of them coincide with Maksimov's multiplication rules. It was proved that between the 15 algebras
 corresponding to the multiplications there are 5 non-isomorphic associative algebras.
 In \cite{Pan} there are other kinds of multiplications for cubic stochastic matrices.

In general, one can fix an $m^3\times m^3\times m^3$- cubic matrix $\left(C_{ijk,lnr}^{uvw}\right)$ as a matrix of structural constants and   give a multiplication of basis cubic matrices as
\[E_{ijk}E_{lnr}=\sum_{uvw}C_{ijk,lnr}^{uvw}E_{uvw}.\]
Then the extension of this multiplication by bilinearity to arbitrary cubic matrices gives a general multiplication on the set $\mathfrak C$.
Under known conditions (see \cite{Jac}) on structural constants one can make this general ACM as a commutative or/and associative algebra.
\end{rk}
\begin{rk}\label{re1}   We note that the number of operations on $I$
 with $\lvert I \rvert=m$ is $m^{m^2}$.
 Let $\tau(m)$ be the number of distinct associative binary operations on a set of size $m$, i.e $\tau(m)=\lvert \mathcal O_m \rvert$.
 It is quickly increasing function of $m$ which can be seen from the following: (\cite{DiJo})\footnote{see also http://math.stackexchange.com/questions/105438/how-many-associative-binary-operations-there-are-on-a-finite-set}
\[
\begin{array}{c|cccccccc}
m&1&2&3&4&5&6&7&8\\\hline
\tau(m)& 1& 8& 113& 3492& 183732& 17061118& 7743056064& 148195347518186
\end{array}
\]
Thus the family of multiplications \eqref{ma} is rich, therefore in this paper we only consider multiplications of the form \eqref{ma}.
 \end{rk}
 \begin{rk} As it was mentioned above, there is one-to-one correspondence between the set $\mathbf A$ of all $n$-dimensional algebras
  (with a fixed basis) and the set of all cubic matrices (of size $n^3$, defined by structural constants). This correspondence can be used to consider $\mathbf A$ as an algebra (of algebras), i.e.
   for any algebra $U\in\mathbf A$ with matrix of structural constants $A=(a_{ijk})$, and  $V\in \mathbf A $ with matrix of structural constants $B=(b_{ijk})$,
    we define $U+V$ as an algebra with matrix of structural constants $A+B$. For $\lambda\in F$ we define the algebra $\lambda U$ with matrix of structural constants $\lambda A$.
Fix a multiplication of cubic matrices, say $\star$,  and define multiplication $U\star V$ to be an algebra
with matrix of structural constants $A\star B$. Thus $(\mathbf A, \star)$ is an algebra of algebras. Some properties of the ACM (with multiplication $ \star$) can be related to properties of $\mathbf A$.
 \end{rk}

\section{Isomorphic ACMs} \label{S:iso}

Let $S_m$ be the group of permutations on $I$.

Take $a\in \mathcal O_m$ and define an action of $\pi\in S_m$ on $a$ (denoted by $\pi a$) as
\[\pi a(i,j)=\pi a(\pi^{-1}(i),\pi^{-1}(j)), \qquad \text{for all} \ \ i,j\in I.\]

\begin{lemma}\label{l1} For any $a\in \mathcal O_m$ and any $\pi \in S_m$ we have $\pi a\in \mathcal O_m$.
\end{lemma}
\begin{proof} It suffices to show that $\pi a$ is an associative operation, i.e.
\begin{equation}\label{ass}
\pi a(\pi a(i,j),k)=\pi a(i,\pi a(j,k)), \quad  \text{for all} \ \ i,j,k\in I.
\end{equation}
For LHS of this equality we have
\[ \pi a(\pi a(i,j),k)=\pi a\Big(\pi \big(a(\pi^{-1}(i),\pi^{-1}(j))\big),k\Big)=\pi a\Big(a\big(\pi^{-1}(i),\pi^{-1}(j)\big),\pi^{-1}(k)\Big).\]
For RHS of \eqref{ass} we get
\[ \pi a(i,\pi a(j,k))=\pi a\Big(i,\pi\big(a(\pi^{-1}(j),\pi^{-1}(k))\big)\Big)=\pi a\Big(\pi^{-1}(i),a\big(\pi^{-1}(j),\pi^{-1}(k)\big)\Big).\]
Since $a$ is associative and $\pi$ is one-to-one we have
\[ \pi a\Big(a\big(\pi^{-1}(i),\pi^{-1}(j)\big),\pi^{-1}(k)\Big) =\pi a\Big(\pi^{-1}(i),a\big(\pi^{-1}(j),\pi^{-1}(k)\big)\Big),\]
i.e. \eqref{ass} holds.
\end{proof}

So, we have an action of the symmetric group, $S_m$, on $\mathcal O_m$, the set of all associative binary operations on $I$.

The \emph{orbit} of $a\in \mathcal O_m$ under the action of $S_m$ is
$ \Orb(a)=\left\{\pi a: \pi\in S_m\right\}$.
Note that the set of orbits of (points $a\in \mathcal O_m$) under the action of $S_m$
 form a partition of $\mathcal O_m$. The associated equivalence relation is
 defined by saying $a\sim b$ if and only if there exists a $\pi\in S_m$ with $\pi a=b$.
 The orbits are then the equivalence classes under this relation;
 two operations $a$ and $b$ are equivalent if and only if their orbits are the same.

The following theorem gives a sufficient condition for the
isomorphness of ACM.

\begin{thm}\label{t1} If $a\sim b$, i.e. there exists a permutation $\pi \in S_m$ such that $\pi b=a$,  that is
\begin{equation}\label{abp}
a(j,n)=\pi^{-1} (b(\pi(j),\pi(n))), \qquad \text{for all} \ \ j,n\in I,
\end{equation}
 then the algebras $\mathfrak C_a$ and $\mathfrak C_b$ are isomorphic.
\end{thm}
\begin{proof} By Lemma~\ref{l1} it follows that if $a\in \mathcal O_m$ and $b\sim a$ then $b\in \mathcal O_m$. Consider $f(E_{ijk})=E_{\pi(i)\pi(j)\pi(k)}$ and  check
  \[ f(E_{ijk}*_a E_{lnr})=f(E_{ijk})*_b f(E_{lnr}).\]
  This equality is equivalent to
  \begin{equation}\label{EE}
  \delta_{kl}E_{\pi(i)\pi(a(j,n))\pi(r)}=\delta_{\pi(k)\pi(l)}E_{\pi(i)b(\pi(j),\pi(n))\pi(r)}.
  \end{equation}
  Since $\pi$ is one-to-one map on $I$, we have $\delta_{kl}=\delta_{\pi(k)\pi(l)}$ for any $k,l$.
  Consequently, if $\delta_{kl}=0$, i.e., $k\ne l$ then \eqref{EE} holds.
If $\delta_{kl}=1$, i.e $k=l$ then \eqref{EE} is reduced to the form
  \[
  E_{\pi(i)\pi(a(j,n))\pi(r)}=E_{\pi(i)b(\pi(j),\pi(n))\pi(r)}.
  \]
  This is true because of condition \eqref{abp}.
\end{proof}
\begin{rk} In \cite{Mak} it was proved that the algebra $\mathfrak C_a$ is associative for each $a\in \mathcal O_m$ and isomorphic to the tensor multiplication of the algebra of
square matrices and the semigroup algebra with respect to the basis $I$ and operation $a$.
\end{rk}

\begin{defn}  An operation $a\in \mathcal O_m$ is called \emph{symmetric} if
its orbit under the action of $S_m$ consists only of $a$ itself, $\Orb(a)=\{a\}$, i.e.
\[
a(j,n)=\pi(a(\pi^{-1}(j),\pi^{-1}(n))), \qquad \text{for all} \ j,n\in I, \ \ \text{for all} \ \pi\in S_m.
\]
(Equivalently, $a$ is fixed point for every $\pi\in S_m$)
\end{defn}
\begin{pro}\label{p1} An operation $a\in \mathcal O_m$ is symmetric if and only if
it is one of the following operation:
\begin{align}
a(i,j)&=j, \qquad \text{for all}  \ \  i,j\in I, \tag{rs}  \label{rs} \\
a(i,j)&=i, \qquad \text{for all}  \ \  i,j\in I. \tag{ls} \label{ls}
\end{align}
\end{pro}
\begin{proof} {\sl Sufficiency}. Assume $a$ is \eqref{rs} (the case \eqref{ls} is similar). We show that
$a$ is symmetric. Take any $\pi \in S_m$.
Then for any $i,j\in I$ we have
\[\pi a(i,j)=\pi a(\pi^{-1}(i),\pi^{-1}(j))=\pi(\pi^{-1}(j))=j=a(i,j),\]
i.e. $a$ is fixed point for any $\pi$.

{\sl Necessity.} Assume $a$ is different from \eqref{rs} and \eqref{ls}. The assertion is trivial for $|I|=m=1$. Case $m=2$
discussed in Example \ref{e2a} below, which shows that proposition is true for $m=2$. Therefore we consider the case $m\geq 3$.
Then the following cases are possible for $a$ which is different from \eqref{rs} and \eqref{ls}:

{\bf Case 1:} There is $i\in I$ such that $a(i,i)=k$ with $k\ne i$.
Consider a permutation $\pi_0$ such that $\pi_0(i)=i$, $\pi_0(k)\ne k$.
Then
 \[\pi_0 a(i,i)=\pi_0 a(\pi_0^{-1}(i),\pi_0^{-1}(i))=\pi_0(k)\ne k=a(i,i),\]
i.e. $a$ is not fixed point of $\pi_0$.

{\bf Case 2:} There are $i,j\in I$ such that $i\ne j$ and $a(i, j) = k$. For $k$ we have several possibilities:

{\bf Subcase 2.1:} $k\ne i, j$.

{\it Subsubcase 2.1.1:} $m=3$. Without loss of generality assume $a(1,2)=3$. Using $a(1,1)=1$ and associativity of $a$
compute $a(1,3)$. We should have $a(a(1,1),3)=a(1,a(1,3)).$
This gives $a(1,3)=a(1,a(1,3))$. Denoting $x=a(1,3)$ we obtain $x=a(1,x)$.
From above mentioned equalities for $a$ we see that $x$ can be 1 or 3 but not 2, i.e. $x\ne 2$.
Define $\pi_2(1)=1$, $\pi_2(2)=3$, $\pi_2(3)=2$, then
\[\pi_2 a(1,3)=\pi_2 a(\pi_2^{-1}(1),\pi_2^{-1}(3))=\pi_2 (a(1,2))=\pi_2(3)=2\ne x=a(1,3),\]
i.e. $a$ is not fixed point of $\pi_2$.

{\it Subsubcase 2.1.2:} $m>3$. For $k\ne i,j$ we consider a
permutation $\tilde\pi=\pi_{ij}$ which satisfies $\tilde\pi(i)=i, \tilde\pi(j)=j$
and $\tilde\pi(k)\ne k$ then we have
\[\tilde\pi a(i,j)=\tilde\pi a(\tilde\pi^{-1}(i),\tilde\pi^{-1}(j))=\tilde \pi(k)\ne k=a(i,j),\]
i.e. $a$ is not fixed point of $\tilde\pi$.

{\bf Subcase 2.2:} $a(i,j)=k\in \{i,j\}$ for any $i$, $j$. After above mentioned cases remain the operations with property $a(i,j)\in \{i,j\}$ for any $i,j$, because if $a(i,j)\ne i,j$ for some $i,j$ then we come back to Subcase 2.1.

{\it Subsubcase 2.2.1:} $a(i,j)=a(j,i)=i$ for some $i\ne j$. (The case $a(i,j)=a(j,i)=j$ is similar.) Consider permutation $\pi_1$ with $\pi_1(i)=j$, $\pi_1(j)=i$ then
\[\pi_1 a(i,j)=\pi_1 a(\pi_1^{-1}(i),\pi_1^{-1}(j))=\pi_1a(j,i)=\pi_1(i)=j\ne i=a(i,j).\]

{\it Subsubcase 2.2.2:} $a(i,j)=i$, $a(j,i)=j$ for any $i,j$. In this case $a$ is \eqref{ls}.

{\it Subsubcase 2.2.3:} $a(i,j)=j$, $a(j,i)=i$ for any $i,j$. In this case $a$ is \eqref{rs}.
This completes the proof.
\end{proof}
\begin{defn}
The symmetric operation \eqref{rs} (resp. \eqref{ls}) mentioned in Proposition~\ref{p1} is called \emph{right-symmetric} (resp. \emph{left-symmetric}).
\end{defn}
\begin{ex}\label{e2a}\hfill

 1. Consider the case $I=\{1,2\}$, then the following 8 operations are
associative (\cite{DiJo}):
\[\begin{array}{cccc}
I& \vline&1&2\\
\hline
1&\vline&1&1\\
2&\vline&1&1
\end{array}\qquad
\begin{array}{cccc}
II& \vline&1&2\\
\hline
1&\vline&1&1\\
2&\vline&1&2
\end{array}\qquad
\begin{array}{cccc}
III& \vline&1&2\\
\hline
1&\vline&1&1\\
2&\vline&2&2
\end{array} \qquad
\begin{array}{cccc}
IV& \vline&1&2\\
\hline
1&\vline&1&2\\
2&\vline&1&2
\end{array}
\]

\,

\[\begin{array}{cccc}
V& \vline&1&2\\
\hline
1&\vline&1&2\\
2&\vline&2&1
\end{array}\qquad
\begin{array}{cccc}
VI& \vline&1&2\\
\hline
1&\vline&1&2\\
2&\vline&2&2
\end{array}\qquad
\begin{array}{cccc}
VII& \vline&1&2\\
\hline
1&\vline&2&1\\
2&\vline&1&2
\end{array} \qquad
\begin{array}{cccc}
VIII& \vline&1&2\\
\hline
1&\vline&2&2\\
2&\vline&2&2
\end{array}
\]

One can check that $I\sim VIII$, $II\sim VI$, $V\sim VII$. Therefore these operations are not symmetric. By Theorem~\ref{t1} the ACMs corresponding to equivalent operations are isomorphic. Operations $III$ and $IV$ are symmetric.

2. In the case $I=\{1,2,3\}$, the set $\mathcal O_3$ of associative operations contains 113 elements (see \cite{DiJo}), here
some examples
 \[\begin{array}{ccccc}
i& \vline&1&2&3\\
\hline
1&\vline&1&2&3\\
2&\vline&2&2&2\\
3&\vline&3&2&2
\end{array}\qquad
\begin{array}{ccccc}
ii& \vline&1&2&3\\
\hline
1&\vline&1&2&3\\
2&\vline&2&3&3\\
3&\vline&3&3&3
\end{array}\qquad
\begin{array}{ccccc}
iii& \vline&1&2&3\\
\hline
1&\vline&1&1&1\\
2&\vline&1&2&3\\
3&\vline&1&3&1
\end{array} \qquad
\begin{array}{ccccc}
iv& \vline&1&2&3\\
\hline
1&\vline&2&2&1\\
2&\vline&2&2&2\\
3&\vline&1&2&3
\end{array}\]

\,

 \[
\begin{array}{ccccc}
v& \vline&1&2&3\\
\hline
1&\vline&1&1&1\\
2&\vline&1&1&2\\
3&\vline&1&2&3
\end{array}\qquad
\begin{array}{ccccc}
vi& \vline&1&2&3\\
\hline
1&\vline&3&1&3\\
2&\vline&1&2&3\\
3&\vline&3&3&3
\end{array}\qquad
\begin{array}{ccccc}
vii& \vline&1&2&3\\
\hline
1&\vline&1&2&3\\
2&\vline&1&2&3\\
3&\vline&1&2&3
\end{array} \qquad
\begin{array}{ccccc}
viii& \vline&1&2&3\\
\hline
1&\vline&1&1&1\\
2&\vline&2&2&2\\
3&\vline&3&3&3
\end{array}
\]
It is easy to check that
$\Orb(i)=\{i,ii,iii,iv,v,vi\}$, consequently by Theorem~\ref{t1} these six operations define isomorphic ACMs.
By Proposition~\ref{p1} we know that the operations $vii$ and $viii$ are symmetric.
\end{ex}

\section{Each ACM is not baric} \label{S:baric}

A \emph{character} $\chi$ for an algebra $\mathcal A$ is a nonzero multiplicative
linear form on $\mathcal A$, that is, a nonzero algebra homomorphism $\chi \colon \mathcal A\to \R$ (see for example, \cite[page 73]{Lyu}).

 A pair $(\mathcal A, \chi)$ consisting of an algebra $\mathcal A$ and a character
$\chi$ on $\mathcal A$ is called a \emph{baric algebra}.

\begin{thm}\label{t2} Each ACM, $\mathfrak C_a$, $a\in \mathcal O_m$, $m\geq 2$, is not baric.
\end{thm}
\begin{proof}
Assume $\chi \colon \mathfrak C_a\to\R$ is a character. Then for each cubic matrix $X=(x_{ijk})\in \mathfrak C_a$
it has the form
 $\chi(X)=\sum_{i,j,k=1}^m \alpha_{ijk} x_{ijk}$. We shall prove that
$\chi(X)\equiv 0$.
Consider it on basis elements:
\[\chi(E_{ijk})=\alpha_{ijk}\in \R, \quad  i,j,k\in I.\]
We should check $\chi(E_{ijk}*_a E_{lnr})=\chi(E_{ijk})\chi(E_{lnr})$.
Using \eqref{ma} the last equality can be written as
\[
 \alpha_{ijk}\alpha_{lnr}=\delta_{kl}\alpha_{ia(j,n)r}, \qquad \text{for all}  \ \ i,j,k,l,n,r\in I.
\]
Equivalently,
\begin{align}
 \alpha_{ijk}\alpha_{lnr} & =0, \qquad  \qquad  \quad    \ \text{for all}  \ \  i,j,k\ne l,n,r\in I. \label{b0}  \\
 \alpha_{ijk}\alpha_{knr} & =\alpha_{ia(j,n)r},  \ \,  \qquad \text{for all}  \ \  i,j,k=l,n,r\in I. \label{b1}
\end{align}
In \eqref{b0} take $i\ne k$ and $l=i$, $n=j$, $r=k$, then we obtain
\begin{equation}\label{ink}
\alpha_{ijk}=0, \qquad  \quad \text{for all} \ \ i\ne k, j\in I.
\end{equation}
Consequently, non-zero coefficient can be only in the form $\alpha_{kjk}$.

Assume  $\alpha_{k_0j_0k_0}\ne 0$, for some $k=k_0$ and $j=j_0$.
Then from \eqref{b0} we get $\alpha_{lnl}=0$, for any $l\ne k_0$ and $n\in I$.
Using \eqref{ink}, from \eqref{b1} for $i=r=k_0$, $k\ne k_0$, we get
\begin{equation}\label{b2}
 0=\alpha_{k_0jk}\alpha_{knk_0}=\alpha_{k_0a(j,n)k_0}, \qquad  \text{for all} \ \   j,k\ne k_0,n\in I.
\end{equation}

Write \eqref{b1} for $i=r=k=k_0$:
\begin{equation}\label{b3}
 \alpha_{k_0jk_0}\alpha_{k_0nk_0}=\alpha_{k_0a(j,n)k_0}, \qquad \text{for all} \  \  j,n\in I.
\end{equation}
By \eqref{b3} for $j=n=j_0$ we get
 \begin{equation}\label{b4}
 \alpha_{k_0j_0k_0}^2=\alpha_{k_0a(j_0,j_0)k_0}.
\end{equation}
Using \eqref{b2}, for $j=n=j_0$ we obtain $\alpha_{k_0a(j_0,j_0)k_0}=0$.
Consequently from \eqref{b4} we get  $\alpha_{k_0j_0k_0}=0$. Thus $\chi(X)\equiv 0$.
\end{proof}

\section{Accompanying algebra of an ACM} \label{S:accom}

 By Theorem~\ref{t2} an ACM is not a baric
 algebra. This is similar to an algebra of bisexual population \cite{LaRo}.
 To overcome such complication, Etherington \cite{Eth3} for
 an algebra of sex linked inheritance introduced the idea of
 treating the male and female components of a population separately.
 In \cite{Hol2} Holgate formalized this concept by introducing sex
 differentiation algebras (which is two-dimensional) and a generalization of baric algebras called dibaric algebras.
 In this section we shall define an analogue of a sex differentiation algebra for an ACM,
 in our case it will be an $m^2$-dimensional algebra.

\begin{defn} Let $\mathfrak A=\langle \eta_{ij}:i,j\in I \rangle$ denote
an $m^2$-dimensional algebra with multiplication
table
\[\eta_{ij}\eta_{kl}=\delta_{jk}\eta_{il}.\]
Then  $\mathfrak A$ is called the \emph{accompanying algebra}.
\end{defn}
The following proposition is obvious.
\begin{pro}
The accompanying algebra $\mathfrak A$ is associative. It is commutative iff $m=1$.
\end{pro}
Now we define a generalization
of a baric and a dibaric algebra.

\begin{defn} An algebra is called accompanied if it admits a
homomorphism onto the accompanying algebra $\mathfrak A$.
\end{defn}

\begin{thm}\label{ta} For each $a\in \mathcal O_m$ the algebra $\mathfrak C_a$ is accompanied.
\end{thm}
\begin{proof} Consider a linear mapping $\varphi \colon \mathfrak C_a\to\mathfrak A$ defined on basis elements by
\[\varphi(E_{inj})=\eta_{ij}, \ \ i,n,j\in I.\]
On one hand we have
\[\varphi(E_{inj}*_aE_{krl})=\delta_{jk}\varphi(E_{ia(n,r)l})=\delta_{jk}\eta_{il}.\]
On the other hand we have
\[\varphi(E_{inj})\varphi(E_{krl})=\eta_{ij}\eta_{kl}=\delta_{jk}\eta_{il}.\]
Thus
\[\varphi(E_{inj}*_aE_{krl})=\varphi(E_{inj})\varphi(E_{krl}),\]
i.e. $\varphi$ generates a homomorphism.

For any $X=(x_{ijk})\in \mathfrak C_a$, i.e.
\[X= \sum_{inj}x_{inj}E_{inj}\] the homomorphism $\varphi$ is defined by
\[\varphi(X)=\sum_{inj}x_{inj}\varphi(E_{inj})=\sum_{ij}\left(\sum_{n}x_{inj}\right)\eta_{ij}.\]

For arbitrary $u=\sum_{ij}u_{ij}\eta_{ij}\in \mathfrak A$ it is
easy to see that $\varphi(X)=u$ if $\sum_{n}x_{inj}=u_{ij}$. Therefore $\varphi$ is onto.
\end{proof}
The accompanying algebra will be useful to study subalgebras of ACM, which we give in a next section.
\begin{rk} We note that baric algebras are useful to describe biological systems of free (one-sex) populations (see \cite{Lyu}).
 Dibaric algebras are related to systems of bisexual (two-sex) populations (see \cite{LaRo,Reed}).
  But in biology there are unusual systems: swordtail fish, the Chironomus midge species, the Platypus has 10
sex chromosomes but lacks the mammalian sex determining
gene SRY, meaning that the process
of sex determination in the Platypus remains
unknown. Zebra fish go through juvenile
hermaphroditism, but what triggers this is unknown.
The Platy fish has W, X, and Y chromosomes.
This allows WY, WX, or XX females or
YY and XY males \cite{Sch}. Such systems may have
``multiple sexes'' instead of having only two.
In such  systems the accompanying algebra will play the role of the sex
 differentiation algebra.  Then an ACM will play the role of an evolution algebra
 of such biological systems (see \cite{LaRo,Lyu,Reed,Tian,Var,Wor} for different kinds of evolution algebras).
\end{rk}
\section{Commutativity and solvability of equations in ACMs} \label{S:comm}

\begin{pro}
 Algebra $\mathfrak C_a$ is commutative for each $a\in \mathcal O_m$ iff $m=1$.
\end{pro}
\begin{proof} In case $m=1$ the algebra is commutative because it is one dimensional, i.e. only with one basis element $E_{111}$.
For $m\geq 2$ taking $k=l$ and $i\ne r$ from \eqref{ma} we get
\[E_{ijk}*_aE_{knr}=E_{ia(j,n)r}\ne 0=E_{knr}*_aE_{ijk}.\]
\end{proof}

An element (cubic matrix) $A$ of $\mathfrak C_a$ is called a \emph{left} (resp. \emph{right}) \emph{zero divisor} if there exists a nonzero $X\in \mathfrak C_a$
 such that $A*_aX = 0$ (resp.  $X*_aA = 0$).

 For a cubic matrix $A=(a_{ijk})\in \mathfrak C_a$ denote
 \[A_{ik}=\sum_{j=1}^m a_{ijk}, \ \ i,j\in I,\]
 and by $\mathbb B$ we denote the square matrix $\mathbb B=(A_{ik})_{i,k=1}^m$, which is called accompanying matrix of the cubic matrix $A$ (see \cite{Mak}).

 \begin{pro}
  Let $\mathfrak C_a$ be ACM over the
 field of real numbers.
 \begin{itemize}
 \item[(i)]  If $a\in \mathcal O_m$ is a right-symmetric (resp. left-symmetric)
 operation then a cubic matrix $A$ is a left (resp. right) zero divisor iff $\det \mathbb B=0$.

 \item[(ii)]  If $a\in \mathcal O_m$ is a left-symmetric (resp. right-symmetric)
 operation then any cubic matrix $A$ is a left (resp. right) zero divisor.
 \end{itemize}
   \end{pro}
 \begin{proof} (i) For the right-symmetric $a$ we have $a(l,n)=n$ for all $l,n$.
 Consequently, the left zero divisibility equation $A*_aX=0$ can be written as (see \eqref{AB})
  \[
 \sum_{l,n: \, a(l,n)=j}\sum_k a_{ilk}x_{knr}=\sum_{l=1}^m\sum_{k=1}^m a_{ilk}x_{kjr}=0, \qquad \text{for all} \ \ i,j,r.
 \]
 Consequently,
\[ \sum_{k=1}^m A_{ik}x_{kjr}=0,  \qquad \text{for all} \ \ i,j,r.\]
Thus $x_{kjr}\ne 0$ iff $\det \mathbb B=0$.
The case of left-symmetric is similar.

(ii)  For the left-symmetric $a$ we have $a(l,n)=l$ for all $l,n$.
 Consequently, the left zero divisibility equation $A*_aX=0$ can be written as (see \eqref{AB})
  \[
 \sum_{l,n: \, a(l,n)=j}\sum_k a_{ilk}x_{knr}=\sum_{n=1}^m\sum_{k=1}^m a_{ijk}x_{knr}=0,  \qquad \text{for all} \ \ i,j,r.
 \]
 Consequently,
\[ \sum_{k=1}^m a_{ijk}X_{kr}=0, \qquad  \text{for all} \ \ i,j,r,\]
where
\[X_{kr}=\sum_{n=1}^mx_{knr}.\]
Hence independently on values $a_{ijk}$ one can choose $x_{knr}$ (some of them should be non-zero)
such that $X_{kr}=0$ for any $k,r$.
The case of right-symmetric is similar.
\end{proof}

 \section{Subalgebras of ACM} \label{S:subal}

A subalgebra of an algebra is a subset of elements that is closed under addition, multiplication, and scalar multiplication.
A left (resp. right) ideal of an algebra is a linear subspace that has the property that any element of the subspace multiplied on the left (resp. right) by any element of the algebra produces an element of the subspace. A two-sided ideal is a subset that is both a left and a right ideal. The term ideal on its own is usually taken to mean a two-sided ideal.

In the following theorem we collect several subalgebras of ACM.

\begin{thm}\label{t3} Let $\mathfrak C_a$, $a\in \mathcal O_m$, be an ACM with basis $\{E_{ijk}: i,j,k\in I\}$.
\begin{itemize}

\item[1.] Let $J_a=\{a(i,j): i,j\in I\}$ be the image of $a\in \mathcal O_m$.
If $\tilde J\subset J_a$ is an $a$-invariant (i.e. $a(\tilde J,\tilde J)\subset \tilde J$)
then for each fixed $i,k$, the set $\mathcal E^{ik}_{\tilde J}=\{E_{ijk}: j\in \tilde J\}$ generates a subalgebra, denoted by
$\mathfrak C_{a,\tilde J}^{(ik)}$.

\item[2.] If $i\ne \tilde i$ or $k\ne \tilde k$, then $\mathfrak C_{a,\tilde J}^{(ik)}\cap \mathfrak C_{a,\tilde J}^{(\tilde i\tilde k)}=\{0\}$, here $0$ is the zero-cubic matrix.

\item[3.] If $\tilde J$ and $\bar J$ are $a$-invariant sets such that $\tilde J\subset \bar J$ then
 $\mathfrak C_{a,\tilde J}^{(ik)}\subset \mathfrak C_{a,\bar J}^{(ik)}$.

\item[4.] If $\tilde J$ and $\bar J$ are $a$-invariant sets such that $\tilde J\cap \bar J=\emptyset$ then
 $\mathfrak C_{a,\tilde J}^{(ik)}\cap \mathfrak C_{a,\bar J}^{(ik)}=\{0\}$.

\item[5.] The set $\{E_{ijk}: i, k\in I, \, j\in J_a\}$ generates an ideal, denoted by  ${\mathfrak I}_{a}$.

\item[6.] The set ${\mathfrak I}_{a}^0=\{X=(x_{inj})\in \mathfrak C_a: \sum_{n}x_{inj}=0,  \ \ \text{for all} \ \ i,j\in I\}$ is an ideal.
\end{itemize}
\end{thm}
\begin{proof} 1. For any two basis elements $E_{ijk}, E_{ink}\in \mathcal E^{ik}_{\tilde J}$,  we have
\[ E_{ijk}*_a E_{ink}=\delta_{ki}E_{ia(j,n)k}\in \mathcal E^{ik}_{\tilde J},\]
because $a(j,n)\in \tilde J$. Thus the set $\mathcal E^{ik}_{\tilde J}$ generates a subalgebra.
Note that the subalgebra $\mathfrak C_{a,\tilde J}^{(ik)}$
is with zero-multiplication iff $i\ne k$.

Items 2--4 are straightforward.

5. It is clear that $\mathfrak I_a$ is a subalgebra. First we prove that
it is a right ideal, i.e. for arbitrary $A\in \mathfrak C_a$ and
$B\in \mathfrak I_a$ one has $A*_aB\in \mathfrak I_a$.
Let $A=\sum_{ijk}a_{ijk}E_{ijk}$ and $B=\sum_{ink: n\in J_a}b_{ink}E_{ink}$.
We have
\[(A*_aB)_{ijr}=\sum_{l,n: \, a(l,n)=j}\sum_k a_{ilk}b_{knr},\]
i.e. $(A*_aB)_{ijr}=0$ if $j\notin J_a$. Consequently, $A*_aB\in \mathfrak I_a$.
Similarly one can see that $B*_aA\in \mathfrak I_a$. Thus $\mathfrak I_a$ is an ideal.

6. We note that  $\mathfrak I_a^0=\ker \varphi=\{X\in \mathfrak C_a: \varphi(X)=0\}$, where $\varphi$ is constructed in the proof of Theorem~\ref{ta}. Therefore it is an ideal.
\end{proof}

For $i\in I$ define the sequence
\begin{equation}\label{in}
i_0=i, \qquad  i_n=a(i_{n-1},i_{n-1}),  \quad n\geq 1.
\end{equation}
Since $I$ is a finite set the sequence $i_n$ may have one of the following forms:
\begin{itemize}
\item[a.] \emph{periodic}: there is a number $p\geq 1$ such that
$i_p=i$. Denote by $p_a(i)\geq 1$, the minimal period of $i$ and $I_a(i)=\{i_0=i, i_1, \dots, i_{p_a(i)-1}\}$.

\item[b.] \emph{convergent}: there is a number $q\geq 1$ such that $i_n=i_q$ for any $n\geq q$.
\end{itemize}

We note that $I_a(i)$ is not $a$-invariant, in general. This can be seen in the following example.

\begin{ex} Take $a\in \mathcal O_3$ as
 \[\begin{array}{ccccc}
a& \vline&1&2&3\\
\hline
1&\vline&1&2&3\\
2&\vline&2&3&1\\
3&\vline&3&1&2
\end{array}\]

Then $I_a(1)=\{1\}$, $I_a(2)=\{2,3\}$ and $I_a(3)=\{3,2\}$. We note that $I_a(1)$ is $a$-invariant, but
$I_a(2)$ is not $a$-invariant, because $a(2,3)=1\notin I_a(2)$.
\end{ex}

Define \emph{plenary powers} of an element $A\in\mathfrak C_a$ as follows:
\begin{align*}
A^{[1]} & =A^{(2)} = A *_a  A, \\
A^{[2]} & =A^{(2^2)} = A^{(2)} *_a  A^{(2)},\\
{} &  \quad \dots\ \ \dots \ \ \dots \\
A^{[n]} & =A^{(2^n)} = A^{(2^{n-1})} *_a  A^{(2^{n-1})}.
\end{align*}
For convenience, we denote $A^{[0]} =A$.

Fix a subset $K\subset I$ and construct the following sequens of sets:
\begin{align*}
J_{a,0}(K) &=K,\\
 J_{a,1}(K) & =J_{a,0}(K)\cup\{a(s,t): s,t\in J_{a,0}(K)\}, \\
{} &  \quad  \dots\ \ \dots \ \ \dots \\
J_{a,n}(K)& =J_{a,n-1}(K)\cup\{a(s,t): s,t\in J_{a,n-1}(K)\}.
\end{align*}
Since $I$ is finite there is $n_0=n_0(K)$ such that
\[J_{a,n}(K)=J_{a,n_0}(K), \ \text{for all} \    n\geq n_0.\]

\begin{pro}\hfill
\begin{itemize}
  \item[1.] For each $i,j\in I$, the sequence $\{E_{jij}^{[n]}\}$ is
periodic (resp. convergent) iff $i_n$ constructed for $i$ by \eqref{in}
is periodic (resp. convergent).
  \item[2.] For each $k,n\in I$, $K\subset I$,  the set $\{E_{kln}: l\in J_{a,n_0(K)}(K)\}$ generates a subalgebra.
\end{itemize}
\end{pro}
\begin{proof}
1. It is easy to see that $E_{jij}^{[n]}=E_{ji_nj}$,  $i,j\in I$.
Thus sequences $E_{jij}^{[n]}$ and $i_n$ have the same behavior.

2. By construction of $ J_{a,n_0(K)}(K) $ it is clear that this set is an $a$-invariant. Therefore the part 1 of Theorem~\ref{t3} completes the proof.
\end{proof}

For $K=\{i\}$ the following example shows that not each $a$-invariant set of $a$ has the form $J_{a,n_0(i)}(i)$.

\begin{ex} Consider the following operation $a\in\mathcal O_3$:
 \[\begin{array}{ccccc}
a& \vline&1&2&3\\
\hline
1&\vline&1&1&1\\
2&\vline&1&2&2\\
3&\vline&1&3&3\\
\end{array}\]
Then $J_{a,n_0(i)}(i)=\{i\}$, $i\in I=\{1,2,3\}$. But the set $\{i,j\}$, for $i\ne j\in \{1,2,3\}$, and $I$
are also $a$-invariant. Thus any subset of $I$ is an $a$-invariant. By above mentioned results the ACM corresponding to this example of operation has at least 7 subalgebras.
\end{ex}
\section*{ Acknowledgements}

 This work was partially supported by Ministerio de Econom\'{\i}a y Competitividad (Spain), grant MTM2013-43687-P
  (European FEDER support included), by Xunta de Galicia, grant GRC2013-045 (European FEDER support
included) and by Kazakhstan Ministry of Education and Science, grant 0828/GF4: ``Algebras, close to Lie:
cohomologies, identities and deformations''. We thank the referee for careful reading of
the manuscript and for a number of useful suggestions which have led to improvement of
the paper.


\end{document}